\documentclass[english, 12pt]{article}

\usepackage{amsfonts}
\usepackage{amsmath}
\usepackage{amssymb}
\usepackage{amsthm}
\usepackage{amscd}
\usepackage{amscd}
\usepackage{amsthm}

\newtheorem{theo}{Theorem}[section]
\newtheorem{prop}{Proposition}[section]
\newtheorem{lem}{Lemma}[section]

\newtheorem{rk}{Remark}[section]

\newtheorem{exa}{Example}[section]
\newtheorem{coro}{Corollary}[section]
\numberwithin{equation}{section}

\def\R{\mathbb{R}}
\def\N{\mathbb{N}}

\def\calE{{\cal{E}}}
\def\calF{{\cal{F}}}

\def\calC{{\cal{C}}}

\newcommand{\Hmu}{H_{\mu}}
\newcommand{\Hnu}{H_{\nu}}

\newcommand{\varp}{\varphi}
\newcommand{\lam}{\lambda}
\newcommand{\Lam}{\Lambda}
\newcommand{\Om}{\Omega}

\newcommand{\alp}{\alpha}

\newcommand{\phiok}{\varp_0^{(k)}}
\newcommand{\phiomu}{\varp_0^{(\mu)}}
\newcommand{\phionu}{\varp_0^{(\nu)}}
\newcommand{\lamok}{\lam_0^{(k)}}
\newcommand{\lamomu}{\lam_0^{(\mu)}}

\newcommand{\lamonu}{\lam_0^{(\nu)}}
\newcommand{\xinu}{\xi^{(\nu)}}
\newcommand{\ximu}{\xi^{(\mu)}}

\begin{document}

\title{Pointwise estimates for the ground states of some classes of positivity preserving operators }

\author{\normalsize Ali Beldi
, Nedra Belhadjrhouma 
 \& Ali BenAmor\footnote{corresponding author} 
}

\date{}
\maketitle
\begin{abstract} We establish pointwise estimates for the ground states  of some classes of positivity preserving operators. The considered operators are negatively perturbed (by measures) strongly local Dirichlet operators. These estimates will be written in terms of the Green's kernel of the considered operators, whose existence will be proved.  In many circumstances our estimates are even sharp so that they recover known results about the subject.  The results will deserve to obtain large time heat kernel estimates for the related operators.
\end{abstract}

\section{Introduction} 
Let $\Om$ be an open connected and bounded subset of the Euclidean space $\R^d$ and $-\Delta_{\Om}$ be the Dirichlet-Laplacian on $\Om$. It is well known that the ground state energy of $-\Delta_{\Om}$, which we denote by $\varp_0$, enjoys the property of being comparable to the function $(-\Delta_{\Om})^{-1}1$.  In other words, if we designate by $G_{\Om}$ the Green's kernel of  $-\Delta_{\Om}$, then
\begin{eqnarray}
\varp_0\sim \int_{\Om}G_{\Om}(\cdot,y)\,dy\ {\rm on}\ \Om.
\end{eqnarray}

 This result was extended to negative perturbations of $-\Delta_{\Om}$ satisfying Kato condition, namely to the ground state $\varp_0^V$ of the operators $-\Delta_{\Om}-V$ where $V$ is a positive measurable function in the Kato-class and under some regularity assumptions imposed on the domain $\Om$ (see for instance the papers of Ba\~{n}uelos \cite{banuelos91}, Davies \cite{davies86} and Davies' book \cite{davies-book} ).\\
Actually,  Ba\~{n}uelos proved (among others) in \cite[Theorem 2]{banuelos91} that if $\Om$ is a nontangentially accessible (NTA) bounded domain  and $V$  in the Kato-class  is such that
\begin{eqnarray}
\lam_0^{V}:=  \inf_{f\in C_c^1(\Om)\setminus\{0\}}\frac{\int_{\Om}|\nabla f|^2\,dx-\int_{\Om}f^2V\,dx}{(\int_{\Om}f^2\,dx)}>0,  
\label{strict}
\end{eqnarray}
is  nondegenerate and has a strictly positive eigenfunction, denoted by $\varp_0^V$, then
\begin{eqnarray}
\varp_0\sim\varp_0^V.
\end{eqnarray}
However,  for (NTA) domains the conditional gauge theorem (which is one of the main ingredients in Ba\~{n}uelos' proof)  holds true  and the Green's functions of $-\Delta_{\Om}$ and that of $-\Delta_{\Om}-V$  ($G_{\Om}$ and $G_{\Om}^V$) are comparable. So that the latter comparison can be written as
\begin{eqnarray}
\varp_0^{V}\sim\int_{\Om}G_{\Om}^V(\cdot,y)\,dy,
\end{eqnarray}
and the latter function is nothing else but the $W_0^{1,2}$ -solution of the equation 
\begin{eqnarray}
-\Delta u-Vu=1\ {\rm on}\ \Om.
\end{eqnarray}
In  \cite{davilla-dupaigne}, D\'avila--Dupaigne improved the result to more general $V$ that do not necessary belong to the Kato class, including for instance
\begin{eqnarray}
 V(x)=(\frac{d-2}{2})^2|x|^{-2}\ {\rm and}\  V(x)=\frac{1}{4}{\rm dist}^{-2}(x,\partial\Om),
\end{eqnarray}
where  $d\geq 3$ and $\Om$ is regular.\\
Those $V$ should satisfy  the conditions that $V\in L^1_{\rm loc}$ and there is $p>2$ such that
\begin{eqnarray}
\inf_{f\in C_c^1(\Om)\setminus\{0\}}\frac{\int_{\Om}|\nabla f|^2\,dx-\int_{\Om}f^2V\,dx}{(\int_{\Om}|f|^p\,dx)^{2/p}}>0,
\label{sob}
\end{eqnarray}
Obviously, condition (\ref{sob}) is  equivalent to an improved  Sobolev type inequality, whose relevance for intrinsic  ultracontractivity  property as well as for the compactness of the resolvent of the operator $-\Delta_{\Om}-V$  was recognized in \cite{davilla-dupaigne}. \\ 
Being inspired by the latter observation, we shall consider, in this paper,  the same problem in a more general framework. Precisely we shall replace the gradient energy form by  a Dirichlet form, $\calE$ with associated positive selfadjoint operator $H$, having the strong local property whose domain lies in some $L^2(X,m)$-space. The potential function $V$ will be however replaced by a positive measure, $\mu$ charging no set having zero capacity.\\
We shall prove that under some realistic assumptions, and especially under the assumptions that some improved Sobolev-Orlicz and Hardy-type inequalities hold true, then the positivity preserving operator related to the semi-Dirichlet form $\calE-\mu$ still shares many interesting features as for the classical case. In particular they have compact resolvent  and non degenerate ground state. Furthermore the ground state is comparable to  the solution, $\ximu$ of the equation $H_{\mu}\ximu=1$ (i.e., comparable  to $H_{\mu}^{-1}1$), where $H_{\mu}$ is the nonnegative selfadjoint   operator related to $\calE-\mu$.\\
Our method is based on a transformation argument (Doob's transformation) that leads first, to construct the operator $H_{\mu}:=H-\mu$  and to the fact that it has compact resolvent and second to  some ultracontractive semigroups (in the particular case where the transformation is done by means of the ground state, if one already nows about  its existence,  this leads to the intrinsic ultracontractivity of the operator $H_{\mu}$).\\
As an intermediate step,  we shall prove that the positivity preserving operators under considerations can be approximated, in the norm resolvent sense  by a sequence of operators whose  ground states can be estimated in a sharp way. This will lead to convergence of ground state energies and ground states and   enables us to carry over the comparison for the approximating operators to the limit operator.\\
To get the estimates for the ground states of the approximating operator we shall use on one side the intrinsic ultracontractivity property and on the other side Moser's  iteration technique as in \cite{davilla-dupaigne}.

\section{The framework and preparing results}
We first shortly describe the framework in which we shall state our results.\\
Let  $X$ be   a  separable locally compact metric space,  $m$ a positive Radon measure on Borel subsets of $X$ such that $m(U)>0,\ \forall\,\emptyset\neq U\subset X$.  All integrals of the type $\int \cdots$ are assumed to be over $X$. The space of real-valued continuous functions having compact support on $X$ will be 
denoted by $C_c(X)$. \\
Let   $\calE$ be a regular symmetric transient Dirichlet form, with domain $\calF:=D(\calE)$ w.r.t. the space $L^2:=L^2(X,m)$.  Along the paper we assume that $\calE$ is strongly local, i.e., $\calE(f,g)=0$, whenever $f,g\in\calF$ and $f$ is constant on the support of $g$ .\\
The local Dirichlet space related to $\calE$ will be denoted by $\calF_{\rm loc}$.  A function $f$ belongs to $\calF_{\rm loc}$ if for every open bounded  subset $\Om\subset X$
there is $\tilde f\in\calF$ such that $f=\tilde f$-a.e. on $\Om$.\\
We recall the known fact $\calE$ induces a positive-valued sets function called capacity.  If  a property holds true up to a set having zero capacity we shall say that it holds quasi-everywhere and we shall write 'q.e.'.\\
It is well know (see \cite{fuku-oshima}) that  every element from $\calF_{\rm loc}$ has a quasi-continuous  (q.c. for short) modification. We shall always implicitly assume that elements from $\calF_{\rm loc}$ has been modified so as to become quasi-continuous.\\
We also designate by $\calF_b:=\calF\cap L^{\infty}(X,m)$ and $\calF_{b,\rm loc}:=\calF_{\rm loc}\cap L_{\rm loc}^{\infty}(X,m)$. From the very definition we derive that both $\calF_b$ and  $\calF_{b,\rm loc}$ are algebras.\\
Given $f,g\in\calF$, we set $\Gamma[f]$ the {\em energy measure} of $f$ and  $\Gamma(f,g)$ the {\em  mutual energy measure} of $f,g$ (see \cite[pp.110-114]{fuku-oshima}). Every strongly local Dirichlet form,  $\calE$ possesses the following representation
\begin{eqnarray}
\calE[f]:=\calE(f,f)=\int_X\,d\Gamma[f],\ \forall\,f\in\calF.
\label{energy}
\end{eqnarray}
The representation goes as follows: for $f\in\calF_b$ its energy measure is defined by
\begin{eqnarray}
\int\phi\,d\Gamma[f]=\calE(f,\phi f)-\frac{1}{2}\calE(f^2,\phi),\ \forall 0\leq\phi\in\calF\cap C_c(X).
\end{eqnarray}
Truncation and monotone convergence allow then to define $\Gamma[f]$ for every $f\in\calF$.\\
Furthermore with the help of the  strong locality property, i.e.,
\begin{eqnarray}
\int_{\{f=c\}}d\Gamma[f]=0,\ \forall\,f\in\calF,
\end{eqnarray}
it is possible to define $\Gamma[f]$ for every $f\in\calF_{\rm loc}$ as follows: for every open bounded subset $\Om\subset X$
\begin{eqnarray}
1_{\Omega}d\Gamma[f]=1_{\Omega}d\Gamma[\tilde f],
\end{eqnarray}
where $\tilde f\in\calF$ and $f=\tilde f$-q.e. on $\Omega$.\\
By polarization and regularity we can thereby define a Radon-measure-valued bilinear form on $\calF_{\rm loc}$ denoted by $\Gamma(f,g)$,  so that
\begin{eqnarray}
\calE(f,g)=\int \,d\Gamma(f,g),\ \forall\,f,g\in\calF_{\rm loc},\ {\rm either}\ f {\rm or}\ g\  {\rm has\ compact\ support}.
\label{energy'}
\end{eqnarray}
The truncation property for $\calE$ reads as follows: For every $a\in\R$, every $f\in\calF_{\rm loc}$, having compact support and every $g\in\calF_{\rm b,loc}$ we have
\begin{eqnarray}
\calE((f-a)_+,g)=\int_{\{f>a\}}\,d\Gamma(f,g)\ {\rm and}\  \calE[(f-a)_+]=\int_{\{f>a\}}\,d\Gamma[f].
\label{truncation}
\end{eqnarray}
Furthermore the following product formula holds true
\begin{eqnarray}
d\Gamma(fh,g)=fd\Gamma(h,g)+hd\Gamma(f,g),\ \forall\,f,g,h\in \calF_{\rm b,loc}.
\label{product}
\end{eqnarray}
By the regularity assumption the latter formula extends to every $f,g,h\in\calF_{\rm loc}$.\\ Another rule that we shall occasionally use is the {\em chain rule} (See \cite[pp.111-117]{fuku-oshima}): For every function $\phi:\R\to\R$ of class $C^1$ with bounded derivative ($\phi\in C_b^1(\R)$), every $f\in \calF_{\rm loc}$ and every $g\in\calF_{\rm b,loc}$, $\phi(f)\in\calF_{\rm loc}$ and
\begin{eqnarray}
d\Gamma(\phi(f),g)=\phi'(f)d\Gamma(f,g).
\label{chain rule}
\end{eqnarray}
Formula (\ref{chain rule}) is still valid for $\phi(t)=|t|^{p/2}$  when restricted to locally quasi-bounded $f$.\\ 
As long as we are concerned with Sobolev-Orlicz inequalities, we will give some material related to the underlying spaces. From now on we shall denote the Lebesgue-Orlicz spaces  $L^{\Phi}(\Omega,\nu)$ simply by  $L^{\Phi}(\nu)$, whereas in the case $\nu=m$ they will be denoted by $L^{\Phi}$. 
We also fix an $N$-function $\Phi:[0,\infty)\to[0,\infty)$, i.e., a convex function such that
\begin{eqnarray}
\Phi(t)=0\iff t=0,\ \lim_{t\to 0}\frac{\Phi(t)}{t}=0,\   \lim_{t\to \infty}{\Phi(t)}= \lim_{t\to \infty}\frac{\Phi(t)}{t}=\infty,
\end{eqnarray}
and denote by $\Psi$ its {\em complementary} function and set
\begin{eqnarray}
\Lambda(s):=\frac{1}{s\Phi^{-1}(1/s)},\ s>0.
\end{eqnarray}
An N-function $\Phi$ is said to be an  {\em admissible}, if   the following integrability condition near zero is satisfied
\begin{eqnarray}
\int_0^{\alp}(s\Lambda(s))^{-1}\,ds<\infty\ {\rm for\ some}\ \alp>0.
\label{integrability}
\end{eqnarray}
We quote that a necessary and sufficient condition for a $N$-function to be admissible is that the  function  $\frac{\Phi^{-1}(t)}{t^2}$ is integrable at  infinity.\\
Among functions that are admissible we cite $N$-functions $\Phi$  satisfying  the $\nabla_2$-condition  ($\Phi\in \nabla_2$ for short), i.e.,  there is $l>1$  and $t_0>0$ such that
\begin{equation}\label{o3}
\Phi(t)\leq \frac{1}{lt}\Phi(lt),\  \forall\,t\geq t_0,
\end{equation}
are admissible. Indeed, by \cite[Corollary 5, p.26]{RR}, if $\Phi\in\nabla_2$ then there is a finite constant $C>0$,  $\epsilon>0$ and $t_0>0$ such that 
\begin{eqnarray}
\Phi(t)\geq Ct^{1+\epsilon},\ \forall\,t>t_0.
\end{eqnarray}
Yielding therefore $\Phi^{-1}(t)\leq C't^{1/{1+\epsilon}}$, for large $t$.\\
Let $\mu$ be a fixed positive Radon measure on Borel subsets of $X$, which does not charge sets having zero capacity. We shall also adopt some assumptions along the paper.\\
The first assumption that we shall adopt, along the paper, is  the following:  there is a function $s\in\calF_{loc}\cap L^2,\ s>0$-q.e. such that
\begin{eqnarray*}
(SUP): &~~& \calE(s,f)-\int_Xsf\,d\mu\geq 0,\ \forall\,0\leq f\in\calF\cap C_c(X).
\end{eqnarray*}
This condition deserves some comments. First the additional assumption $s\in L^2$ is automatically satisfied if either $X$ is relatively compact.\\
Second,  assumption (SUP) asserts the existence of a nonnegative supersolution of the operator $H-\mu$ and is, according to \cite{fitz00,benamor-belhaj}, almost equivalent to the occurrence of the following Hardy's inequality
\begin{eqnarray}
\int f^2\,d\mu\leq \calE[f],\ \forall\,f\in\calF.
\label{hardy}
\end{eqnarray}
By 'almost equivalent' we mean that  if (SUP)  holds true  then inequality (\ref{hardy}) holds true as well. However, if (\ref{hardy}) occurs then for every $\delta\in(0,1)$ there is $s\in\calF$ such that 
\begin{eqnarray}
  \calE(s,f)-\delta\int_Xsf\,d\mu\geq 0,\ \forall\,0\leq f\in\calF\cap C_c(X).
\label{reciproque}
\end{eqnarray} 
We shall maintain, throughout the paper, that the following improved Sobolev-Orlicz inequality holds true: there is a finite constant $C_S>0$  such that
\begin{eqnarray*}
(ISO): &~~& \parallel \!f^2\!\parallel_{L^{\Phi}}
\leq C_S\big(\mathcal{E}[f]-\int f^2\,d\mu\big),\ \forall\,f\in \mathcal{F}.
\end{eqnarray*}
For  discussions about connections between (ISO) (especially in the case where $\mu=0$) and various types of Logarithmic-Sobolev inequalities we refer the reader to \cite{cipriani-sob, wang09}.\\ 
In conjunction with $\Phi$, there is another function which will play a decisive role in the paper and which we denote by $\phi_1:=\phi_1(t)=t\Psi^{-1}(t),\ \forall\,t\geq 0$. We assume from no on that the function $\phi_1$ is {\em admissible}.\\
The following lemma indicates that the latter condition is fulfilled in many situations, in particular for $\Phi(t)=1/pt^p,\ t\geq 0$ and $1<p<\infty$.

\begin{lem} Assume that $\Phi\in\nabla_2$ and that $\phi_1$ is convex. then $\phi_1$ is admissible.
\label{admissible}
\end{lem}

\begin{proof} From the fact that $\Phi$ is an $N$-function we deduce
\begin{eqnarray}
\phi_1(t)=0\iff t=0,\ \lim_{t\to 0}\phi_1(t)/t=0\ {\rm and}\ \lim_{t\to\infty}\phi_1(t)/t=\infty,
\end{eqnarray}
which together with the convexity assumption yields that $\phi_1$ is an $N$-function.\\
The integrability condition: From  the known inequality for conjugate Young functions
\begin{eqnarray}
t\leq \Phi^{-1}(t)\Psi^{-1}(t)\leq 2t,\ \forall\,t\geq 0,
\end{eqnarray}
in conjunction with the fact that  $\Phi\in\nabla_2$, we obtain that there is  $a>0$,  $\epsilon>0$ and $t_0>0$ such that 
\begin{eqnarray}
\phi_1(t)\geq at^{2-\frac{1}{1+\epsilon}},\ \forall\,t>t_0.
\end{eqnarray}
Thus for large $t$ we have $t^{-2}\phi_1^{-1}(t)\leq at^{\frac{1+\epsilon}{1+2\epsilon}-2}$, and the latter function is integrable at infinity,  yielding the admissibilty of $\phi_1$.
\end{proof}

We also have an inclusion relation between the spaces $L^{\Phi}$ and $L^{\phi_1}$.
\begin{lem} The space  $L^{\Phi}$ embeds continuously into $ L^{\phi_1}$.
\label{inclusion}
\end{lem}

\begin{proof} From Young's inequality
\begin{eqnarray}
tr\leq \Phi(t)+\Psi(r),\ \forall\,r,t\geq 0,
\end{eqnarray}
we get
\begin{eqnarray}
t\Psi^{-1}(t)=\phi_1(t)\leq \Phi(t)+t,\ \forall\,t\geq 0.
\end{eqnarray}
Taking the behavior of $\Phi$ at infinity into account: $\lim_{t\to\infty}\Phi(t)/t=\infty$, we conclude that there is $T>0$ such that
\begin{eqnarray}
\phi_1(t)\leq 2\Phi(t),\ \forall\,t>T.
\end{eqnarray}
Since $m(X)<\infty$ we conclude that $L^{2\Phi}\subset L^{\phi_1}$, with continuous inclusion. The result follows by observing that the spaces $L^{2\Phi}$ and $L^{\Phi}$ have equivalent norms.  
\end{proof}
Set $H$ the positive selfadjoint operator associated to $\calE$ via Kato's representation theorem. For every $t>0$ we set $T_t:=e^{-tH}$ the semigroup associated to the operator $H$.\\
In the next theorem we will collect some spectral properties of the operator $H$
 on the light of the improved Sobolev-Orlicz inequality.

\begin{theo} For every $t>0$, the operator $T_t$ is ultracontractive. It follows that
\begin{itemize}
\item[i)] The operator $H$ has compact resolvent.
\item [ii)] Set $\lam_0$ the smallest eigenvalue of $H$. Then $\lam_0$ is 
nondegenerate, i.e., there is $\psi_0$  (the ground state) such that $\psi_0>0$ - q.e. and $ker (H-\lam_0)=\R\psi_0$. Furthermore $\psi_0$ is  quasi-bounded.
\end{itemize}
\label{spec-prop1}
\end{theo} 
 
\begin{proof} 
Since $\phi_1$ is admissible, and  $L^{\Phi}\subset L^{\phi_1}$, continuously,  with the help of  \cite[Theorem 3.4]{benamor07}, we derive that $T_t:=e^{-tH}$ is ultracontractive. Thus it  has a nonnegative absolutely continuous essentially bounded  kernel $p_t,\ \forall\,t>0$. Hence since $m(X)<\infty$, we conclude that $T_t$ is a Hilbert-Schmidt operator, yielding that $H$ has compact resolvent.\\
On the other hand owing to  \cite[Proposition 1.4.3, p.24]{davies-book}, the Dirichlet form $\calE$ is irreducible, which implies  that the smallest eigenvalue of $H$, which we denote by $\lam_0$, is simple and has a q.e. nonnegative  normalized eigenfunction $\psi_0$. The quasi-boundedness of $\psi_0$ follows from the ultracontractivity property of $T_t$ and the proof is finished.
\end{proof}

\begin{rk} We have already mentioned in the proof of Theorem\ref{spec-prop1} that the Dirichlet form $\calE$ is irreducible, which implies together with the fact that $\calE$ is strongly local, that $X$ is connected ( see \cite{sturm94}).
\end{rk}

From the fact that $T_t$ is a Hilbert-Schmidt operator, we also derive  that the inverse operator $H^{-1}$ possesses a Green kernel $G_X$ which is positive,  symmetric and measurable.\\
We shall assume, throughout the paper, that the following Hardy-type inequality holds true: There is a constant $0<C_H<\infty$ such that
\begin{eqnarray*}
(HI) &~~~& \int_{}\frac{f^{2}}{\psi_0^2}dm \leq C_H\mathcal{E}[f],\  \forall\,f\in \mathcal{F}.
\end{eqnarray*}

\begin{prop} There exists a finite constant $C_G>0$ such that
\begin{eqnarray}
G_X(x,y)\geq C_G\psi_0(x)\psi_0(y),\ a.e..
\end{eqnarray}
\label{green}
\end{prop}

\begin{proof} Set  $\calE^{\psi_0}$ the quadratic form defined on $L^2(\psi_0^2dm)$ by
\begin{eqnarray}
D(\calE^{\psi_0})=\big\{f:\psi_0f\in\calF\big\},\ \calE^{\psi_0}[f]=\calE[\psi_0f]
,\ \forall\,f\in D(\calE^{\psi_0}).
\end{eqnarray}
Then $\calE^{\psi_0}$ is a  Dirichlet form. Indeed, $\calE^{\psi_0}$ is related (via Kato's representation theorem) to the operator $H^{\psi_0}:=\psi_0^{-1}H\psi_0$, so that $e^{-tH^{\psi_0}}=\psi_0^{-1}e^{-tH}\psi_0$, which is Markovian.\\
In this step we will prove that $D(\calE^{\psi_0})$ embeds continuously into the space $L^{\phi_1}$.\\
We claim that
\begin{eqnarray}
\parallel f^2\parallel_{L^{\phi_{1}}(\psi_0^2dm)}
\leq 2(C_S+C_H)\calE^{\psi_0}[f],\ \forall\,f\in D(\calE^{\psi_0}).
\end{eqnarray}
Indeed, by H\"older's inequality we find

\begin{eqnarray}
  \int\psi_0^2\phi_1(\frac{f^2}{(C_S+C_H)\calE^{\psi_0}[f]})\,dm&=&\int \frac{\psi_0^2f^2}{(C_S+C_H)\calE^{\psi_0}[f]}\psi^{-1}(\frac{f^2}{(C_S+C_H)
\calE^{\psi_0}[f]})\,dm\nonumber\\
&\leq& 2\parallel \frac{\psi_0^2f^2}{(C_S+C_H)\calE^{\psi_0}[f]}\parallel_{L^{\phi}}\nonumber\\
&\cdot&\parallel\Psi^{-1}(\frac{f^2}{(C_S+C_H)\calE^{\psi_0}[f]})\parallel_{L^{\Psi}}.
\end{eqnarray}
By $(ISO)$, we have
$$\parallel\frac{\psi_0^2 f^2}{(C_S+C_H)\calE^{\psi_0}[f]}\parallel_{L^{\phi}}\leq\frac{C_S}{C_S+C_H}.
$$
On the other hand, by inequality (HI), we get
\begin{eqnarray}
\int_{}\Psi(\Psi^{-1}(\frac{f^2}{(C_S+C_H)\calE^{\psi_0}[f]}))\,dm
=\int_{}\frac{f^2}{(C_S+C_H)\calE^{\psi_0}[f]}dm\leq\frac{C_S}{C_S+C_H}\leq 1,
\end{eqnarray}
yielding
\begin{eqnarray}
\parallel\Psi^{-1}(\frac{f^2}{(C_S+C_H)\calE^{\psi_0}[f]})\parallel_{L^{\Psi}}
\leq 1.
\end{eqnarray}
Finally,  from the definition of the Luxemburg's norm we achieve

\begin{eqnarray}
\parallel f^2\parallel_{L^{\phi_{1}}(\psi_0^2dm)}
\leq 2(C_S+C_H)\calE^{\psi_0}[f],\ \forall\,f\in D(\calE^{\psi_0}[f]),
\end{eqnarray}
and the claim is proved.\\
Now since $\phi_1$ is admissible, using another time \cite[Theorem 3.4]{benamor07}, we derive that the semigroup $S_t:=e^{-tH^{\psi_0}},\ t>0$ is ultracontractive and has an absolutely continuous essentially bounded  kernel $k_t$, furthermore
\begin{eqnarray}
k_t(x,y)=\psi_0(x)\psi_0(y)p_t(x,y),\ a.e..
\end{eqnarray}
By standard way (see \cite[p.112]{davies-book}), we conclude that there is $T>0$ such that, $\forall\,t>T$, 
\begin{eqnarray}
\frac{1}{2}e^{-\lam_0t}\psi_0(x)\psi_0(y)\leq p_t(x,y),\ a.e..
\end{eqnarray}
Integrating on $(0,\infty)$, yields
\begin{eqnarray}
G_X(x,y)\geq\frac{e^{-\lam_0T}}{2\lam_0}\psi_0(x)\psi_0(y),\ a.e.,
\end{eqnarray}
which finishes the proof.
\end{proof}

Through the proof of Proposition \ref{green}, we have  proved that the operator $H$ is in fact intrinsicly ultracontractive.\\
From now on we set $\dot\calE_{\mu}$ the form defined by
\begin{eqnarray*}
D(\dot\calE_{\mu})=\calF,\ \dot\calE_{\mu}[f]=\mathcal{E}[f]-\int f^2\,d\mu,\ \forall\,f\in\calF.
\end{eqnarray*}
Since $\calE$ is a Dirichlet from, then $\dot\calE_{\mu}$ is a {\em semi-Dirichlet form}, i.e.,
\begin{eqnarray}
\forall\,f\in D(\dot\calE_{\mu})\Rightarrow |f|\in D(\dot\calE_{\mu}).
\end{eqnarray}
We will prove in the following lines that the form $\dot\calE_{\mu}$ is closable.\\  Let us stress that since  the measure $\mu$ is not assumed to be a small perturbation we can not conclude directly its closabilty by using the KLMN theorem. To that end we  give first some auxiliary results.\\
We say that a function $u\in\calF_{loc}$ is a supersolution of $H-\mu$ if 
\begin{eqnarray}
\calE(u,f)-\int uf\,d\mu\geq 0,\ \forall\,0\leq f\in\calF_{\rm loc}\cap C_c(X).
\end{eqnarray}
 
\begin{lem} Let $s\geq 0$   q.e.  be a positive supersolution of $H-\mu$. Then
\begin{eqnarray}
s(x)\geq C_G\psi_0(x)\int \psi_0(y)s(y)\,d\mu(y),\ q.e.
\end{eqnarray}
\label{lb-sup}
\end{lem} 

\begin{proof} Let  $f\in\calF\cap C_c(X)$ be nonnegative. Set $U={\rm supp}\,f$ and let $u\in\calF$ be such that $u=s$ q.e. on $U$ (such $u$ exists because $s\in\calF_{\rm loc}$). Since $|u|\in \calF$ and $|u|=u=s$  q.e. on $U$ ($s\geq 0$ q.e.), we may and do suppose that $u\geq 0$ q.e. Owing to the definition of $s$ we derive
\begin{eqnarray}
0\leq\calE(s,f)-\int sf\,d\mu=\calE(u,f)-\int uf\,d\mu=\calE(u,f)-\calE(K^{\mu}u,f),
\end{eqnarray}
where
\begin{eqnarray}
 K^{\mu}u:=\int G_X(\cdot,y)u(y)\,d\mu(y),
\label{potentialoperatot}
\end{eqnarray}
is the potential of the measure $u\mu$. Thus $u-K^{\mu}u$ is a potential, obtaining thereby that  $u-K^{\mu}u\geq 0$ q.e.. Since $u=s$ q.e. on $U$ and $u$ is positive q.e., and whence $\mu$ a.e., we get with the help of the lower bound for the Green function in term of the ground state $\psi_0$ (see Prop.\ref{green})
\begin{eqnarray}
s(x)\geq C_G\psi_0(x)\int\psi_0(y)u(y)\,d\mu(y),\ {\rm q.e.\ On}\ U.
\end{eqnarray}
Now let $(U_k)$ be a sequence of compact sets exhausting $X$ and $(u_k)\subset\calF$ such that $u_k\geq 0$ q.e. and $u_k=s$ q.e. on $U_k$ for every integer $k$. Since $U_k\subset U_l,\ \forall\,l\geq k$, we get $u_k=u_l$ q.e. On $U_k,\ \forall\,l\geq k$. Furthermore $u_k\uparrow s$ q.e. So that the estimate established above yields
\begin{eqnarray}
s(x)\geq C_G\psi_0(x)\int\psi_0(y)u_l(y)\,d\mu(y),\ {\rm q.e.\ On}\ U_k\ \forall\,l\geq k.
\end{eqnarray}
Passing to the limit w.r.t. $l$ yields
\begin{eqnarray}
s(x)\geq C_G\psi_0(x)\int\psi_0(y)s(y)\,d\mu(y),\ {\rm q.e.\ On}\ U_k\ \forall\,k.
\end{eqnarray}
Regarding $(U_k)$ exhausts $X$, the lemma is proved.
\end{proof}

Let $s>0$   q.e. be a supersolution of $H-\mu$ (such an $s$ exists by assumption (SUP)).  As a second step toward proving the closability of the form $\dot\calE_{\mu}$ we will prove that the $s$-transform of $\dot\calE_{\mu}$ is in fact a pre-Dirichlet form.\\
We designate by $\dot\calE_{\mu}^s$ (the $s$-transform of $\dot\calE_{\mu}$) the form defined by
\begin{eqnarray}
D(\dot\calE_{\mu}^s):=\calF^s=\{f\colon\,sf\in\calF\}\subset L^2(s^2dm),\ \dot\calE_{\mu}^s[f]=\dot\calE_{\mu}[sf],\ \forall\,f\in\calF^s.
\end{eqnarray}
The following result was mentioned in \cite{fitz00} with a probabilistic proof.  For the convenience of the reader we will give an alternative analytic proof.\\

\begin{lem} The form $\dot\calE_{\mu}^s$ is a pre-Dirichlet form in  $L^2(s^2dm)$. It follows in particular that $\dot\calE_{\mu}$ is a closable and its closure is a  semi-Dirichlet form.
\label{closability}
\end{lem}

\begin{proof} Following Fitzsimmons \cite{fitz08}, we set 
\begin{eqnarray}
\calC^s:=\{f\colon\,f\in\calF_b,\  f\in L^2(s^2dm),\ f\in L^2(\Gamma[s]),\,s\in L^2(\Gamma[f])\}\subset L^2(s^2dm)
\end{eqnarray}
and $Q$ the form defined by
\begin{eqnarray}
Q:=D(Q)=\calC^s,\ Q[f]=\int s^2d\Gamma[f],\ \forall\,f\in D(Q).
\end{eqnarray}
We claim first, that for every $f\in\calC^s$, $sf\in\calF$ (so that $f\in\calF^s$) and 
\begin{eqnarray}
\dot\calE_{\mu}^s[f]=Q[f]+2\int sf\,d\Gamma(s,f)+\int f^2\,d\Gamma[s]-\int f^2s^2\,d\mu.
\end{eqnarray}
Indeed, let $f\in\calC^s$. Then $sf\in\calF_{loc}$ and by the chain rule we get for every open bounded subset $U\subset X$,
\begin{eqnarray}
\int_U\,d\Gamma[sf]=\int_U s^2\,d\Gamma[f]+2\int_U sf\,d\Gamma(s,f)+\int_U f^2\,d\Gamma[s].
\end{eqnarray} 
Owing to the properties of $f$, and exhausting $X$ by open subsets, we get by Schwartz's inequality together with 
monotone convergence
\begin{eqnarray}
\calE_{\mu}[sf]&=&\int\,d\Gamma[sf]= Q[f]+2\int sf\,d\Gamma(s,f)+\int f^2\,d\Gamma[s]\nonumber\\
&\leq& Q[f]+2(\int f^2\,d\Gamma[s])^{1/2}(\int s^2\,d\Gamma[f])^{1/2} +\int f^2\,d\Gamma[s] <\infty,
\end{eqnarray}
yielding that $sf\in\calF$ and  the corresponding formula for $\dot\calE_{\mu}^s[f]$.\\
As a second step we define another form, which we denote by $q$,  as follows
\begin{eqnarray}
D(q)=\calC^s,\ q[f]=2\int sf\,d\Gamma(s,f)+\int f^2\,d\Gamma[s]-\int f^2s^2\,d\mu,\ \forall\,f\in D(q).
\end{eqnarray}
Then $q$ is well defined. Since for every $f\in\calC^s$ also  $f^2\in\calC^s$, we get by the preceding step that $sf^2\in\calF$. Thus, owing to the fact that $s$ is a supersolution  we obtain

\begin{eqnarray}
q[f]=\calE(s,sf^2)-\int s(sf^2)\,d\mu\geq 0,\ \forall\,f\in\calC^s.
\end{eqnarray}

We shall prove that there is a positive measure, $\tilde\mu$ charging no  set having zero capacity  such that 
$$ 
q[f]=\int f^2\,d\nu,\ \forall\,f\in \calC^s .
$$
Let $f\in\calC^s$,  having compact support and $f\geq 0$\ a.e. . Set 
\begin{eqnarray}
L(f):= \calE(s,sf)-\int s(sf)\,d\mu=\int \,d\Gamma(s,sf)-\int s(sf)\,d\mu\geq 0,
\end{eqnarray}
because $s$ is a supersolution. Since $f\mapsto d\Gamma(s,sf)$ is a Radon measure charging no set having zero capacity, we derive that $L$ is actually a positive Radon measure charging no set having zero capacity: There is a positive Radon measure $\tilde\mu$, charging no set having zero capacity such that
\begin{eqnarray}
L(f)=\int f\,d\tilde\mu.
\end{eqnarray}  
Observing that $L(f^2)=q[f]$ we get $q[f]=\int f^2\,d\tilde\mu$, for every $f\in\calC^s$  having compact support  and whence for every  $f\in\calC^s$.

Now Set 
\begin{eqnarray}
S:=Q+q.
\end{eqnarray}
Then $S$ coincides with $\dot\calE_{\mu}^s$ restricted to $\calC^s$.\\
On one hand according to \cite[Theorem 3.10]{fitz08}, the form $Q$ is closable and its closure $\overline{Q}$ is a Dirichlet form having the strong local property. On the other hand since the measure $q$ is positive and absolutely continuous w.r.t. the capacity, then according to \cite{stollmann92}, the form $S$ is closable, yielding the closability of $\dot{\calE_{\mu}^s}$ and whence of $\dot{\calE_{\mu}}$. The fact that the closure of $\dot\calE_{\mu}$ is a semi-Dirichlet form is derived from the fact that $\calE_{\mu}$ is itself a semi-Dirichlet form.\\
Let us denote by $\bar{S}$, respectively $\bar{Q}$ the closure of $S$, respectively of $Q$ and by $L_S$, respectively $L_Q$ the selfadjoint operator associated to $S$, respectively $Q$. Then since $\bar{S}\geq\bar{Q}\geq 0$  we derive that
\begin{eqnarray}
0\leq e^{-t L_S}\leq e^{-t L_Q},\ \forall\,t>0.
\end{eqnarray}
Owing to the fact that $\bar{Q}$ is a Dirichlet form we get that the operator $e^{-t L_Q}$ is Markovian for every $t>0$, and whence $e^{-t L_S},\ t>0$ is Markovian as well or equivalently $\bar{S}$ is a Dirichlet form. Clearly $\bar{S}$ is local and the proof is finished.\\
\end{proof}

We  quote that the improved Sobolev-Orlicz inequality (ISO) has no relevance for the closability of the form $\dot\calE_{\mu}$.\\
From now on we denote by $\calE_{\mu}^s$, respectively $\calE_{\mu}$, the closure of  $\dot\calE_{\mu}^s$, respectively of $\dot\calE_{\mu}$. Actually, we deduce  from the last proof that since $\calC^s $ is a common core for both $\bar{S}$ and $\calE_{\mu}^s$, then $\bar{S}=\calE_{\mu}^s$.\\ 
The form $\calE_{\mu}$ is a densely defined nonnegative form, and is even a semi-Dirichlet form. Let $H_{\mu}$ be the self-adjoint operator associated to $\calE_{\mu}$. Then $H_{\mu}$ is positivity preserving and by inequality (ISO) is invertible with bounded inverse, which we denote by $H_{\mu}^{-1}$. Henceforth we denote by $H_{\mu}^s$ the operator related to the form $\calE_{\mu}^s$ and by $e^{-tH_{\mu}},\ t>0$, respectively $T_t^s:=e^{-tH_{\mu}^s},\ t>0$ the semigroup of operators related to $H_{\mu}$, respectively $H_{\mu}^s$.\\

\begin{theo} Let $s$ be a function satisfying assumption (SUP). Then for every $t>0$, the operator $T_t^s$ is a Hilbert-Schmidt operator. It follows, in particular  that $e^{-tH_{\mu}},\ t>0$ is a Hilbert-Schmidt operator as well and the operator $H_{\mu}$ has a compact resolvent.
\label{compact-resolvent} 
\end{theo}
Let us emphasize that the latter theorem is the only place where we used the supplementary assumption $s\in L^2$.

\begin{proof} By similar arguments to those used in the proof of Prop.\ref{green}, we derive that there is a finite constant $C>0$ such that 
\begin{eqnarray}
\parallel f^2\parallel_{L^{\phi_{1}}(s^2dm)}
\leq C\calE_{\mu}^s[f],\ \forall\,f\in{\cal{C}}_2^s.
\end{eqnarray}
Having in mind that ${\cal{C}}_2^s$ is a core for $\calE_{\mu}^s$, the latter inequality extends to every element from the space $\calF^s$.   Since $\phi_1$ is admissible, $\calE_{\mu}^s$ is a Dirichlet form (by Lemma \ref{closability}) and $s\in L^2$, we get according to \cite[Theorem 3.4]{benamor07} that $T_t^s$ is a Hilbert-Schmidt operator for every $t>0$. Now the rest of the proof follows directly by realizing that $e^{-tH_{\mu}}=sT_t^ss^{-1}$.

\end{proof}

From now on we denote by $\lam_0^{(\mu)}$ the smallest eigenvalue of the operator $H_{\mu}$. 
 We proceed to prove that $\lam_0^{(\mu)}$ is nondegenerate, i.e. the associate eigenspace has dimension one and may be generated by a nonnegative eigenfunction. To that end we shall approximate the operator $H_{\mu}$, in the norm resolvent sense, by a sequence of operators having the mentioned property.\\
Let  $(\mu_k)$ be an increasing sequence of positive measures charging no sets having zero capacity such that $\mu_k\uparrow\mu$ and there is a constant $0<\kappa_k<1$ such that for every $k\in\N$ we have
\begin{eqnarray}
\int_{}f^{2}\,d\mu_k \leq\kappa_k\calE[f],\ \forall\,f\in \mathcal{F}_.
\label{infinitesimal}
\end{eqnarray}
For example the sequence $\mu_k=(1-\frac{1}{k})\mu$ satisfy the above conditions.\\
By the assumption $0<\kappa_k<1$, we conclude that the  following forms
\begin{eqnarray*}
D(\calE_{\mu_k})=\calF,\ \calE_{\mu_k}[f]=\mathcal{E}[f]-\int_{\Om}f^2\,d\mu_k,\ \forall\,f\in\calF,
\end{eqnarray*}
are closed in $L^2$. For every integer $k$, we shall designate by $H_k$ the self-adjoint operator related to $\calE_{\mu_k}$.\\
According to general results about  convergence of sequences of monotone quadratic forms (see \cite{kato}), one can realize that $H_k\to H_{\mu}$, in the strong resolvent sense as $k\to\infty$. We shall improve this observation in the following way:

\begin{lem} The operators $H_k$ have compact resolvents and 
\begin{eqnarray}
\lim_{k\to\infty}\|H_{k}^{-1}-H_{\mu}^{-1}\|=0.
\end{eqnarray}
\label{norm-conv}
\end{lem}

\begin{proof} Observe that $0\leq H_k^{-1}\leq H_{\mu}^{-1}$. Now  the first statement follows from the fact that $H_{\mu}^{-1}$  is  compact  and the second one   follows from the known fact that $H_{\mu}^{-1}$ is compact together with the norm resolvent convergence \cite[Theorem 3.5, p.453]{kato}.
\end{proof}

The latter lemma will have a great influence on the strategy that we shall follow. This is illustrated through the following:

\begin{coro}
\begin{itemize}
\item [i)]Let $\lam_0^{(k)}$, respectively $\lam_0^{(\mu)}$ be the smallest eigenvalue of the operator $H_k$ respectively $H_{\mu}$. Then $\lim_{k\to\infty}|\lamok-\lamomu|=0$.
\item[ii)]Let $P^{(k)}$, respectively $P^{(\mu)}$ be the eigenprojection of the eigenvalue $\lam_0^{(k)}$, respectively of the eigenvalue $\lam_0^{(\mu)}$. Then
\begin{eqnarray}
\lim_{k\to\infty}\|P^{(k)}-P^{(\mu)}\|=0.
\label{projection}
\end{eqnarray}
It follows, in particular, that if $\lam_0^{(k)}$ is nondegenerate for large $k$, then so is  $\lam_0^{(\mu)}$ and conversely.
\end{itemize}
\label{approximation}
\end{coro}

\begin{proof} (i): Follows from the inequality $|\frac{1}{\lam_0^{(k)}}-\frac{1}{\lam_0^{(\mu)}}|\leq\|H_{k}^{-1}-H_{\mu}^{-1}\|$ and  Lemma \ref{norm-conv}.\\
(ii): Follows from Lemma \ref{norm-conv} and the fact that if $P$ and $Q$ are two orthogonal projections such that $\|P-Q\|<1$, then their respective ranges have the same dimension \cite[Theorem 6.32, p.56]{kato}.
\end{proof}

\begin{lem} Let $\nu$ be a positive Radon measure on Borel subset of $X$ such that there is a constant $0<C_{\nu}<1$ with
\begin{eqnarray}
\int f^2\,d\nu\leq C_{\nu}\calE[f],\ \forall\,f\in\calF.
\end{eqnarray}
Let $\calE_{\nu}$ be the form defined by
\begin{eqnarray*}
D(\calE_{\nu})=\calF,\ \calE_{\nu}[f]=\mathcal{E}[f]-\int f^2\,d\nu,\ \forall\,f\in\calF,
\end{eqnarray*}
and $\lam_0^{(\nu)}$ be the smallest eigenvalue of $\calE_{\nu}$.
\begin{itemize}
\item[i)] Let $\varp\geq 0\ q.e.$ be an eigenfunction associated to  $\lam_0^{(\nu)}$. Then
\begin{eqnarray}
\varp(x)\geq\big(C_G\lam_0^{(\nu)}\int\psi_0(y)\varp(y)\,dm(y)\big)\psi_0(x),\ q.e..
\label{estimate-eigenfunction}
\end{eqnarray}
It follows that $\varp>0\ q.e.$.\\
\item[(ii)] The eigenvalue $\lam_0^{(\nu)}$ is nondegenerate and has a positive normalized ground state which we shall denote by $\varp_0^{(\nu)}$.
\end{itemize}
\label{nondegenerate}
\end{lem}

\begin{proof} i): Let $\varp\geq 0\ q.e.$ be any  eigenfunction associated to  $\lam_0^{(\nu)}$.\\ 
Set
\begin{eqnarray}
K^{\nu}\varp=\int G_X(\cdot,y)\varp(y)\,d\nu,\  K\varp=\int G_X(\cdot,y)\varp(y)\,dm,\  
 u=\varp-K^{\nu}\varp-\lam_0^{(\nu)}K\varp.
\end{eqnarray}
Owing to the fact that $\varp$ lies in $\calF$ and hence lies in $L^2(\nu)$, we obtain that the signed measure   $\varp\nu$ has finite energy integral with respect to the Dirichlet form $\calE$, i.e., 
\begin{eqnarray}
\int |f\varp|\,d\nu\leq \alp(\calE[f])^{1/2},\ \forall,\ f\in \calF\cap C_c(X),
\end{eqnarray}
and therefore  $K^{\nu}\varp\in\calF$. Thus $u\in\calF$ and satisfies the identity
\begin{eqnarray}
\calE(u,g)&=&\calE(\varp,g)-\int\varp g\,d\nu
-\lam_0^{(\nu)}\int\varp g\,dm\nonumber\\
&=&\calE_{\nu}(\varp,g)-\lam_0^{(\nu)}\int\varp g\,dm=0,\ \forall\,g\in\calF.
\end{eqnarray}
Since $\calE$ is positive definite  we conclude that $u=0\, a.e.$ (and hence q.e.), which yields
\begin{eqnarray}
\varp&=&K^{\nu}\varp+\lam_0^{(\nu)}K\varp
\geq \lam_0^{(\nu)}K\varp
=\lam_0^{(\nu)}\int G_X(\cdot,y)\varp(y)\,dm(y)\nonumber\\
&&\geq \big(C_G\lam_0^{(\nu)}\int\psi_0(y)\varp(y)\,dm(y)\big)\psi_0,\ q.e.,
\end{eqnarray}
where thw latter inequality is obtained from (\ref{green}). \\
ii): Let $\varp$ be an  eigenfunction associated to $\lam_0^{(\nu)}$. Since $\calE_{\nu}$ is a semi-Dirichlet form, then  $|\varp|\in\calF$ and  minimizes the ratio
$$
\big\{\frac{\calE_{\nu}[f]}{\int f^2\,dm}:\ f\in\calF\setminus\{0\}\big\}.
$$
Thus  $|\varp|$ is  an eigenfunction associated to $\lam_0^{(\nu)}$ as well and by assertion (i), $|\varp|>0\ q.e.$.\\
Set $\tilde\varp:=|\varp|-\varp$. Then $\tilde\varp$ satisfies $H_{\nu}\tilde\varp=\lam_0^{(\nu)}\tilde\varp$. Now, either $\tilde\varp=0\ a.e.$ which would imply that $\varp=|\varp|\ a.e.$ or $\tilde\varp$ is a non-negative eigenfunction associated to $\lam_0^{(\nu)}$. In the latter case we derive from assertion (i) that $\tilde\varp>0\ q.e.$ or equivalently $|\varp|>\varp\ q.e.$. We have thereby proved that every eigenfunction associated to $\lam_0^{(\nu)}$ has a constant sign, from which (ii) follows.

\end{proof}

On the light of Corollary \ref{approximation} together with Lemma \ref{nondegenerate}, we conclude that $\lamomu$ is nondegenerate as well and we can get even more:

\begin{lem} Let $\varp_0^{(k)}$ be the  normalized a.e. positive eigenfunction associated to $\lam_0^{(k)}$. Then there is a subsequence $(\varp_0^{(k_j)})$ such that
$$\lim_{j\to\infty}\|\varp_0^{(k_j)}-\varp_0^{(\mu)}\|_{L^2}=0,$$
where $\varp_0^{(\mu)}$ is the normalized a.e. positive eigenfunction associated to $\lam_0^{(\mu)}$.
\label{conv-eigenfunction}
\end{lem}

\begin{proof} Since the sequence $(\varp_0^{(k)})$ is bounded in $L^2$, there is a subsequence,  which we still denote by $(\varp_0^{(k)})$,  and $h\in L^2$ such that $\varp_0^{(k)}\rightharpoonup h$.  Let $P$ be the eigenprojection associated to $\lam_0^{(\mu)}$. Since $P$ is a rank one operator, we get $P\varp_0^{(k)}\to Ph$ in $L^2$. Thus
\begin{eqnarray}
P_k\varp_0^{(k)}=\varp_0^{(k)}=(P_k-P)\varp_0^{(k)}+P\varp_0^{(k)}\to Ph,
\end{eqnarray}
and $\|Ph\|_{L^2}=1$.\\
On the other   hand we may and  shall suppose that $Ph\geq 0\ a.e.$ (by mean of a subsequence if necessary). Now Setting $\varp_0^{(\mu)}:=Ph$, and recalling that ${\rm Ran}P={\rm ker}(H_{\mu}-\lam_0^{(\mu)})$ ( by the fact that $\dim{\rm Ran}P=1$) we get that $\varp_0^{(\mu)}$ is an eigenfunction corresponding to $\lam_0^{(\mu)}$ and $\varp_0^{(\mu)}\geq 0\ a.e.$. Finally Corollary \ref{approximation} together with the lower bound for  $\varp_0^{(k)}$ given by Lemma \ref{estimate-eigenfunction}, lead to
\begin{eqnarray}
\varp_0^{(\mu)}\geq \big(C_G\lam_0^{(\mu)}\int_{\Om}\psi_0(y)\varp_0^{(\mu)}(y)\,dm(y)\big)\psi_0,\ a.e.,
\end{eqnarray}
yielding $\varp_0^{(\mu)}>0\  a.e.$, which completes the proof.
\end{proof}

At the end of this section we resume our strategy. Define
\begin{eqnarray}
\xi^{(k)}:=H_k^{-1}1,\ \xi^{(\mu)}=H_{\mu}^{-1}1.
\end{eqnarray}

\begin{theo} Let $\varp_0^{(k)}, \xi^{(k)}, \lam_0^{(\mu)}, \varp_0^{(\mu)}, \xi^{(\mu)}$ be as above. Assume that for every $k\in\N$ there is a constant $0<\Gamma_k<\infty$  such that $\lim_{k\to\infty}\Gamma_k=\Gamma\in (0,\infty)$ and
\begin{eqnarray}
\Gamma_k^{-1} \xi^{(k)}\leq \varp_0^{(k)}\leq\Gamma_k\xi^{(k)},\ a.e.\ \forall\,k\ {\rm large}.
\label{estimate-strategy}
\end{eqnarray}
Then $\Gamma^{-1}\xi^{(\mu)}\leq\varp_0^{(\mu)}\leq\Gamma\xi^{(\mu)},\ a.e.$.
\label{strategy}
\end{theo}

\begin{proof} By the norm resolvent convergence of $H_k$ towards $H_{\mu}$ (Lemma \ref{norm-conv}), we obtain $\xi^{(k)}=H_k^{-1}1\to H_{\mu}^{-1}1=\xi^{(\mu)}$ in $L^2(\Om)$ and we can assume that $\lim_{k\to\infty}\xi^{(k)}=\xi^{(\mu)},\ a.e.$. Now the result follows from the assumptions of the theorem together with Lemma \ref{conv-eigenfunction}.
\end{proof}

Our main task in the next section is to establish  estimate (\ref{estimate-strategy}).

\section{Estimating the ground state}

In this section we fix:\\
\begin{itemize}
\item[i)] A positive measure $\nu$ satisfying assumptions of Lemma \ref{nondegenerate}.
\item[ii)] Two real-valued, measurable  a.e. positive and essentially bounded functions $V$ and $F$ on $X$ such that either $V\neq 0$ or $F\neq 0$.
\end{itemize}
Let $w\in\calF$. We say that $w$ is a solution of the equation
\begin{eqnarray}
\Hnu w=Vw+F,
\end{eqnarray}
if
\begin{eqnarray}
\calE_{\nu}(w,f)=\int fVw\,dm+\int fF\,dm,\ \forall\,f\in\calF.
\end{eqnarray}
Let $w>0$ q.e. be a  solution (if any) of the equation $\Hnu w=Vw+F$. Define $Q^w$ the form:

\begin{eqnarray}
D(Q^w)=\big\{f:\ wf\in \calF\},\ Q^w[f]=\calE_{\nu}^{w}[f]-\int f^2w^2V\,dm,\ \forall\,f\in D(Q^w).
\end{eqnarray}
Then by the same arguments used in the proof of Lemma \ref{closability}, we deduce that $Q^w$ is a Dirichlet form on $L^2(w^2dm)$ having the local property. Moreover since $w\in L^2$, then the vector space
\begin{eqnarray}
{\cal{C}}^w:=\big\{f\colon\,f\in\calF_b,\ w\in L^2(d\Gamma[f])\big\},
\end{eqnarray}
is a core for $Q^w$.\\ 
We claim that
\begin{eqnarray}
Q^w[f]=\int w^2\,d\Gamma[f]+\int f^2Fw\,dm,\ \forall\,f\in {\cal{C}}^w.
\label{transform}
\end{eqnarray}
Indeed, from the product formula for the energy measure, we derive
\begin{eqnarray}
Q^w[f]=\int w^2\,d\Gamma[f]-\int f^2w^2\,d\nu-\int f^2w^2V\,dm
+\int \,d\Gamma(w,wf^2),\ \forall\,f\in {\cal{C}}^w.
\label{AVD}
\end{eqnarray}
Using the fact that $w$ is a solution of equation (\ref{w-solution}), we get for every $f\in {\cal{C}}^w$,   $wf^2\in {\cal{C}}^w$ and 
\begin{eqnarray}
\calE_{\nu}(w,wf^2)=\int Vf^2w^2\,dm+\int Fwf^2\,dm=\int \,d\Gamma(w,wf^2) -\int f^2w^2\,d\nu,
\end{eqnarray}
and substituting in Eq. (\ref{AVD})  we get the claim.\\
We also note that the operator $w^{-1}(H_{\nu}-V)w$ is the self-adjoint operator in $L^2(w^2dm)$ associated to the Dirichlet form $Q^w$.\\
Henceforth, we define
\begin{eqnarray} 
C':=C_G^{-2}\big(\int \psi_0(y)w(y)V(y)\,dm + \int \psi_0(y)F(y)\,dm\big)^{-2},
\end{eqnarray}
and
\begin{eqnarray}
C:=\max(C_HC',C_HC'\lam_0).
\end{eqnarray}

\begin{theo} Let $V, F$ be as in the beginning of this section. Let $w\in\calF$, $w>0$ q.e.  be a  solution of the equation
\begin{eqnarray}
\Hnu w=Vw+F,
\end{eqnarray}
\label{w-solution}
Set
$$
A:= (C+2C_s)\big(1+2C_S\|1\|_{L^{\Psi}}\big).
$$
Then
$$(ISO1)\qquad \parallel f^2\parallel_{L^{\phi_{1}}(w^2dm)}\leq A\big(Q^{w}[f]+\int Vf^2w^2\,dm\big),\ \forall\,f\in D(Q^w).
$$
\label{main1}
\end{theo}

The proof of Theorem \ref{main1} relies upon auxiliary  results which we shall state in three lemmata.

\begin{lem} Let $w$ be as in Theorem \ref{main1}. Then the  following inequality holds true
\begin{eqnarray}
w\geq C_G\psi_0\big(\int \psi_0(y)V(y)w(y)\,dm+\int \psi_0(y)F(y)\,dm\big)\ q.e..
\end{eqnarray}
\label{w-lower}
\end{lem}
\begin{proof} As in the proof of Lemma \ref{nondegenerate} we show that $w$ satisfies
$$
w-K^{\nu}w=KVw+KF,
$$
yielding, whith the help of the lower estimate (\ref{green}) for the Green function $G_{X}$,
$$
w\geq KVw+KF\geq C_G\psi_0\int\psi_0(y)V(y)w(y)\,dm+C_G\psi_0\int\psi_0(y)F(y)\,dm\ q.e..
$$
\end{proof}

\begin{lem} Let $w$ be as in Theorem \ref{main1}. Then
\begin{equation}\label{s2}
 \int f^2\,dm\leq C\int {w}^2d\Gamma[f]\,dm+C\int w^2f^2\,dm,\ \forall\,  f\in{\cal{C}}^w.
\end{equation}
\label{L2-estimate}
\end{lem}

\begin{proof}
At this stage we use Hardy's inequality (HI), which states that there
is a constant $C_H>0$ depending only on $X$ such that
\begin{equation}\label{s3} \int\frac{u^{2}}{\psi_0^{2}}dm \leq C_H\int d\Gamma[u],\ \forall\,u\in \mathcal{F}.
 \end{equation}
Let $f\in {\cal{C}}^w$. Taking $ u=f\psi_0$ in inequality  (\ref{s3}) yields
\begin{eqnarray}\label{S3}
\int f^2\,dm&=&\int\frac{f^2\psi_0^2}{\psi_0^2}\,dm\leq
 C_H\int d\Gamma[f\psi_0]\nonumber\\
   &=& C_H\int\psi_0^2d\Gamma[f]+2C_H\int\psi_0f d\Gamma(\psi_0,f)+C_H\int f^2d\Gamma(\psi_0,\psi_0)\nonumber \\
  &=& C_H\int\psi_0^2d\Gamma[f]+C_H\int d\Gamma(\psi_0,\psi_0f^2).
\end{eqnarray}
Thanks to the fact that $\psi_0$ is an eigenfunction associated to $\lam_0$, we achieve
\begin{equation}\label{Ss}
\int d\Gamma(\psi_0,\psi_0f^2)= \lam_0
\int f^2 \psi_0^2\,dm .
\end{equation}
Combining  (\ref{Ss}) with (\ref{S3}) we obtain
\begin{equation}\label{psi}
 \int f^2\,dm\leq C_H\int{\psi_0}^2d\Gamma[f]+C_H\lam_0\int{\psi_0}^2f^2\,dm,\ \forall\,f\in {\cal{C}}^w.
\end{equation}
Having the lower bound for $w$ given by Lemma \ref{w-lower} in hand, we establish
\begin{eqnarray}
\int f^2\,dm\leq C_HC'\int {\psi_0}^2d\Gamma[f]+C_HC'\lam_0\int {\psi_0}^2f^2dm,\ 
\forall\,f\in {\cal{C}}^w.
\end{eqnarray}

\end{proof}

\begin{lem} Let $w$ be as in Theorem \ref{main1}. Set
\begin{eqnarray}
\Lam_1=1+\frac{C_HC'}{2},\ \Lam_2=\frac{\|F\|_{\infty}^2}{2}+\frac{C_HC'\lam_0}{2},
\end{eqnarray}
$C'$ being the constant appearing Proposition \ref{L2-estimate}. Then
\begin{equation}\label{c1}Q^{w}[f]\leq\Lambda_1\int w^{2}d\Gamma[f]+\Lambda_2\int w^2f^2\,dm
,\ \forall\,f\ {\cal{C}}^w.
\end{equation}
\end{lem}
\begin{proof}
We have already established that
\begin{eqnarray}
Q^w[f]=\int w^2\,d\Gamma[f]+\int f^2Fw\,dm,\ \forall\,f\in {\cal{C}}^w.
\end{eqnarray}
Making use of H\"older's and Young's inequality together with inequality (\ref{s2}) we obtain

\begin{eqnarray*}
Q^{w}[f]&\leq& \int w^{2}d\Gamma[f]+\big(\int f^2\,dm\big)^{\frac{1}{2}}
\big(\int f^2F^2w^2\,dm\big)^{\frac{1}{2}}\nonumber\\
&\leq&\Lam_1\int w^{2}d\Gamma[f]+\Lam_2\int f^2w^2\,dm,\ \forall\,f\in {\cal{C}}^w,
\end{eqnarray*}
which finishes the proof

\end{proof}

\begin{proof} {\em of Theorem \ref{main1}}. We observe first that
\begin{eqnarray}
Q^{w}[f]+\int Vf^2w^2\,dm=\calE_{\nu}^w[f]:=\calE_{\nu}[wf], \forall\,f\in {\cal{C}}^w.
\end{eqnarray}
So that due to the fact ${\cal{C}}^w$ is a core for the form $Q^w$ it suffices to prove inequality (ISO1) on ${\cal{C}}^w$.\\
For $f\in{\cal{C}}^w $, set $\lam:=A\calE_{\nu}^w[f] $. By H\"older's inequality for Orlicz norms, we get for every $f\in {\cal{C}}^w$,
\begin{eqnarray*}
  \int_{}w^2\phi_1(\frac{f^2}{\lam})\,dm&=& \int  \frac{w^2f^2}{\lam}\Psi^{-1}(\frac{f^2}{\lam})\,dm\\
   &\leq& 2\parallel \frac{w^2f^2}{\lam}\parallel_{L^{\Phi}}
   \parallel\Psi^{-1}(\frac{f^2}{\lam})\parallel_{L^{\Psi}}.\\
\end{eqnarray*}
By $(ISO)$,  we have
\begin{eqnarray}
2\parallel\frac{w^2 f^2}{\lam}\parallel_{L^{\Phi}}\leq\frac{2}{\lam}C_s\calE_{\nu}^w[f]  \leq 1.
\end{eqnarray}
On the other hand we have, according to   Lemma \ref{L2-estimate}
\begin{eqnarray} 
\int \Psi(\Psi^{-1}(\frac{f^2}{\lam}))\,dm
=\int \frac{f^2}{\lam}\,dm\leq \frac{C}{\lam}\big(\int w^2d\Gamma(f,f)\,dm
+\int w^2f^2\,dm\big).
\end{eqnarray}

Applying another time H\"older's inequality we get
\begin{eqnarray}
\int (fw)^2\,dm\leq\|1\|_{L^{\Psi}}\|(fw)^2\|_{L^{\Phi}}\leq C_S\|1\|_{L^{\Psi}}
\calE_{\nu}^w[f],\ \forall f\in {\cal{C}}^w.
\label{holder}
\end{eqnarray}
Recalling that $\calE_{\nu}^w[f]\geq\int w^2\,d\Gamma[f]$, we achieve
\begin{eqnarray}
\int \frac{f^2}{\lam}\,dm\leq \frac{C}{\lam}\big(1+2C_S\|1\|_{L^{\Psi}}\big)\calE_{\nu}^w[f]\leq 1 ,\ \forall f\in {\cal{C}}^w.
\end{eqnarray}
Thus
\begin{eqnarray}
\|\Psi^{-1}(\frac{f^2}{\lam})\|_{L^{\Psi}}\leq 1,
\end{eqnarray}
and whence
\begin{eqnarray}
\int w^2\phi_1(\frac{f^2}{\lam})\,dm\leq 1,\ \forall f\in {\cal{C}}^w,
\end{eqnarray}
 and the theorem is proved, according to the definition of the Orlicz norm.
\end{proof}

For every $t>0$ we designate by $T_t^w$ the semigroup associated to the form $Q^w$ in the space $L^2(w^2dm)$. We are yet ready to prove the ultracontractivity of   $T_t^w$.\\
To that end we collect some preparing notations.  We recall the expression of the constant $A$
\begin{eqnarray}
A:= (C+2C_s)\big(1+2C_S\|1\|_{L^{\Psi}}\big).
\end{eqnarray}
Let $\Lambda$ be the function defined by
\begin{eqnarray}
\Lambda(s):=\frac{1}{s\phi_1^{-1}(1/s)}, \forall\,s>0,
\end{eqnarray}
and $\gamma$ be the solution of the equation
\begin{eqnarray}
t:=8A\int_{0}^{\gamma(t)}\frac{1}{s\Lambda(s)}\,ds.
\end{eqnarray}
We finally denote by
\begin{eqnarray}
\beta(t):=\frac{4}{\gamma(t)}.
\end{eqnarray}

\begin{theo} Let $V, F$ and $w$ be as in Theorem \ref{main1}. Then $T_t^w$ is ultracontractive for every $t>0$ and
\begin{eqnarray}
\|T_t^w\|_{L^1(w^2dm),L^{\infty}}\leq \beta(t/2)e^{\|V\|_{\infty}t},\ \forall t>0.
\end{eqnarray}
\label{UC}
\end{theo}

\begin{proof}From Theorem \ref{main1}, we derive
\begin{eqnarray}
\parallel f^2\parallel_{L^{\phi_{1}}(w^2dm)}\leq A\big(Q^{w}[f]+\|V\|_{\infty}\int_{\Om}f^2w^2\,dm\big),\ \forall\,f\in D(Q^w).
\end{eqnarray}
Since $\phi_1$ is admissible, we get according to \cite{benamor07}, that the semi-group $T_t^w$ is ultracontractive for every $t>0$ and
\begin{eqnarray}
\|T_t^w\|_{L^1(\Om,w^2dm),L^{\infty}}\leq \beta(t/2)e^{\|V\|_{\infty}t},\ \forall t>0.
\end{eqnarray}
\end{proof}

We shall apply Theorem \ref{main1}, to the special cases $V=0, F=1$ which corresponds to $w=\xi^{(\nu)}$.

\begin{theo} Let $\nu$ be as in th beginning of this section. Then the following pointwise upper bound for $\phionu$  holds true
\begin{eqnarray}
\phionu\leq\big(\beta(t/2)e^{t\lamonu}\big) \xi^{(\nu)} ,\ a.e.\ \forall\,t>0.
\end{eqnarray}
\label{comparison1}
\end{theo}

\begin{proof}  Applying Theorem \ref{main1} to the case $V=0, F=1$, so that we may and do choose $w=\xi^{(\nu)}$, yields that the semi-group $T_t^{\xi^{(\nu)}}$ is ultracontractive and $\frac{\phionu}{\xi^{(\nu)}}$ is an eigenfunction for $T_t^{\xi^{(\nu)}}$ associated to the eigenvalue $e^{-t\lamonu},\ \forall\,t>0$. Thus
\begin{eqnarray}
\|\frac{\phionu}{\xi^{(\nu)}}\|_{\infty}&\leq& e^{t\lamonu}\|T_t^{\xi^{(\nu)}}\|_{L^2((\xi^{(\nu)})^2dm),L^{\infty}}\nonumber\\
& &\leq \beta(t/2)e^{t\lamonu},\ \forall t>0,
\end{eqnarray}
and
\begin{eqnarray}
\phionu\leq \xi^{(\nu)}\beta(t/2)e^{t\lamonu},\ a.e.\ \forall\,t>0.
\end{eqnarray}
which was to be proved.
\end{proof}
While for the upper pointwise estimate we exploited the idea of intrinsic ultracontractivity, for the reversed estimate we shall however, make use of Moser's iteration technique as utilized in \cite{davilla-dupaigne}. To that end and being inspired by D\'avila--Dupaigne \cite{davilla-dupaigne}, we shall  further assume that the function $\phi_1$ satisfies the following growth condition: there is $\epsilon>0$,  and a finite constant $a>0$ such that
\begin{eqnarray}
\phi_1(t)\geq at^{1+\epsilon},\ \forall\,t\geq 0.
\label{growth-condition}
\end{eqnarray}
Regarding the equivalence between the norms of the Orlicz spaces $L^{\phi_1}$ and $L^{a^{-1}\phi_1}$ we may and shall assume that $a=1$.\\
Before stating the result we need a short preparation. We denote by
\begin{eqnarray} 
I_{\nu}:=(\calF,\calE)\to L^2(\nu),\ f\mapsto f,\ K^{\nu}:=I_{\nu}I_{\nu}^*,\  {\rm and}\ K:=H^{-1}.
\end{eqnarray}
We recall \cite{benamor04} that $C_{\nu}=\|K^{\nu}\|$. \\
An elementary computation yields that 
\begin{eqnarray}
(I_{\nu}K)^*:L^2(\nu)\to L^2,\ f\mapsto \int G_X(\cdot,y) f(y)\,d\nu,
\end{eqnarray}
\begin{eqnarray}
 K^{\nu}f:L^2(\nu)\to L^2(\nu),\ K^{\nu}f= \int G_X(\cdot,y)f(y) \,d\nu,\ \forall\,f\in L^2(\nu).
\end{eqnarray}
Furthermore according to \cite[formula (24)]{brasche-pota}
\begin{eqnarray}
H_{\nu}^{-1}=K+(I_{\nu}K)^*(1-K^{\nu})^{-1}I_{\nu}K.
\label{resolvent-formula}
\end{eqnarray}

\begin{lem} The following estimate holds true
\begin{eqnarray}
\xinu\leq \frac{1}{1-C_{\nu}}K1\leq \frac{1}{1-C_{\nu}}\ a.e..
\end{eqnarray}
\label{last-estimate}
\end{lem}

\begin{proof} By assumption, we have
\begin{eqnarray}
(1-C_{\nu})\calE[f]\leq \calE_{\nu}[f]\leq  \calE[f],\ \forall\,f\in\calF,
\end{eqnarray}
so that 
\begin{eqnarray}
K\leq H_{\nu}^{-1}\leq (1-C_{\nu})^{-1}K.
\end{eqnarray}
According to the identity (\ref{resolvent-formula}), the operator $H_{\nu}^{-1}$ possesses a positive symmetric kernel, which we denote by $G_X^{\nu}$. Thus the self-adjoint positive operator   $ (1-C_{\nu})^{-1}K- H_{\nu}^{-1}$ has a symmetric kernel and by \cite[Lemma 1.4.1, p.24 ]{fuku-oshima}, the kernel is positive $m\times m$ a.e.. Whence $G_X^{\nu}\leq  (1-C_{\nu})^{-1}G_X$ a.e., from which follows
\begin{eqnarray}
\xinu=H_{\nu}^{-1}1\leq  (1-C_{\nu})^{-1}K1\ a.e..
\end{eqnarray}
On one hand,  since $H$ is a Dirichlet  operator then for every $\alp>0$, we have $(H+\alp)^{-1}1\leq 1$ a.e. and on the other hand $\lim_{\alp\to 0}(H+\alp)^{-1}=K$, strongly, yielding   $\lim_{\alp\to 0}(H+\alp)^{-1}1=K1$ in $L^2$ so that by means of a subsequence we get 
$K1\leq 1$ a.e. and the proof is finished.

\end{proof}

\begin{theo} For every $t>0$, the  following estimate holds true
comparison holds true

\begin{equation}\label{GD}
\xi^{(\nu)}\leq (AC+1)(C(\nu,t)+1)\varphi_0^{(\nu)},\ a.e., 
\end{equation}
where 
\begin{eqnarray}
C(\nu,t):=\beta(t/2)e^{t\lamonu}\ \forall\,t>0.
\end{eqnarray}
\label{comparison2}
\end{theo}

\begin{proof} Consider the ratio
\begin{eqnarray}
\rho:=\frac{\xi^{(\nu)}}{\varphi_0^{(\nu)}}.
\end{eqnarray}

By \cite[Lemma 2.2, Lemma 2.1]{benamor-belhaj}, the function $\frac{1}{\phionu}$ lies in the space $\calF_{b,\rm loc}$. Thus according to Lemma \ref{last-estimate}, $\rho\in\calF_{b,\rm loc}$. Now  using the chain rule  together with the equations satisfied by the ground state $\varphi_{0}^{(\nu)}$ and the function $\xi^{(\nu)}$, we find,  for every $f\in\calF_{\rm loc}$ having compact support,
\begin{eqnarray}
\int(\varphi_{0}^{(\nu)})^{2}
d\Gamma(f,\rho)&=&\int_{}d\Gamma(\varphi_0^{(\nu)}f,\xi^{(\nu)})-\int_{}f
d\Gamma(\varphi_0^{(\nu)},\xi^{(\nu)})-\int_{}\xi^{(\nu)}d\Gamma(\varphi_0^{(\nu)},f)\nonumber\\
&=&\int
f\varphi_0^{(\nu)}dm+\int_{}\varphi_0^{(\nu)}f\xi^{(\nu)}d\nu-\int_{}d\Gamma(\varphi_0^{(\nu)},f\xi^{(\nu)})\nonumber\\
&=& \int 
f\varphi_0^{(\nu)}dm+\int_{}\varphi_0^{(\nu)}f\xi^{(\nu)}d\nu-
\big(\lambda_0^{\nu}\int_{}\varphi_0^{(\nu)}f\xi^{(\nu)}dm\nonumber\\
&+&\int_{}\varphi_0^{(\nu)}f\xi^{(\nu)}d\nu\big)=\int_{}\varphi_0^{(\nu)}fdm-
\lambda_0^{(\nu}\int_{}\varphi_0^{(\nu)}f\xi^{(\nu)}dm.
\end{eqnarray}

Let $U$ be a compact subset of $X$. Testing the latter  equation with $f=1_{U}\rho^{2j-1}$, $j\geq 1$, ($f\in \calF_{b,\rm loc}$ by Lemma \ref{last-estimate}),  we deduce
\begin{eqnarray}
\int_{U}(\varphi_{0}^{(\nu)})^{2}
d\Gamma(\rho^{2j-1},\rho)=\int_{U}\rho^{2j-1}(\varphi_0^{(\nu)}-\lambda_0^{(\nu)}\xi^{(\nu)}\varphi_0^{(\nu)})dm,
\end{eqnarray}
which yields, due to the positivity of both functions $\varphi_{0}^{(\nu)}$ and  $\xi^{(\nu)}$
\begin{eqnarray}
\frac{2j-1}{j^2}\int_{U}(\varphi_{0}^{(\nu)})^{2}d\Gamma[\rho^{j}]&=&\int_{U}\rho^{2j-1}
(\varphi_0^{(\nu)}-\lambda_0^{(\nu)}\xi^{(\nu)}\varphi_0^{(\nu)})dm\nonumber\\
&\leq&\int_{U}\rho^{2j-1}\varphi_0^{(\nu)}dm.
\end{eqnarray}

According to Theorem \ref{comparison1}, we obtain

\begin{equation}\label{GD1}
  \int_U(\varphi_{0}^{(\nu)})^{2}d\Gamma[\rho^{j}]\leq
  C(\nu,t)j\int_{U}(\varphi_{0}^{(\nu)})\rho^{2j}\,dm.
\end{equation}

Using H\"older inequality and Lemma \ref{L2-estimate} ( with $V=\lam_0^{(\nu)}, F=0, w=\varphi_{0}^{(\nu)}$ and $f=1_U\rho^j$), it follows from (\ref{GD1})
that

\begin{eqnarray}
\int_U(\varphi_{0}^{(\nu)})^{2}d\Gamma[\rho^{j}]&\leq
&C(\nu,t)j\big(\int_U(\varphi_{0}^{(\nu)})^{2}\rho^{2j}dm\big)^{\frac{1}{2}}(\int_{}\rho^{2j}dm)^{\frac{1}{2}}\nonumber\\
&\leq &
C^{1/2}C(\nu,t)j\big(\int_U(\varphi_{0}^{(\nu)})^{2}\rho^{2j}dm\big)^{\frac{1}{2}}\big(\int_U(\varphi_{0}^{(\nu)})^{2}\,d\Gamma[\rho^{j}]\nonumber\\&+&
\int_U(\varphi_{0}^{(\nu)})^{2}\rho^{2j}dm\big)^{\frac{1}{2}}.
\end{eqnarray}

By Young's inequality, we obtain
\begin{equation}
\int_U(\varphi_{0}^{(\nu)})^{2}d\Gamma[\rho^{j}]\leq
\frac{1}{2}(CC^2(\nu,t)j^2+1)\int_U(\varphi_{0}^{(\nu)})^{2}\rho^{2j}dm+\frac{1}{2}\int_{U}(\varphi_{0}^{(\nu)})^{2}d\Gamma[\rho^{j}],
\end{equation}

so that 

\begin{equation}
\int_U(\varphi_{0}^{(\nu)})^{2}d\Gamma[\rho^{j}]\leq
 (CC^2(\nu,t)j^2+1)\int_{U}(\varphi_{0}^{(\nu)})^{2}\rho^{2j}dm.
\end{equation}

By (\ref{transform}), with $V=\lam_0^{(\nu)}, F=0, w=\varphi_{0}^{(\nu)}$ and $f=1_U\rho^j$, we get from (ISO1)

\begin{eqnarray}
\qquad\parallel 1_U\rho^{2j}\parallel_{L^{\phi_{1}}((\varphi_{0}^{(\nu)})^{2}dm)}\leq A\big(\int_{U}(\varphi_{0}^{(\nu)})^{2}d\Gamma[\rho^{j}]+
\lambda_0^{(\nu)}\int_{U}\rho^{2j}(\varphi_{0}^{(\nu)})^{2}\,dm\big),
\end{eqnarray}

which yields

\begin{eqnarray}
\qquad \parallel 1_U\rho^{2j}\parallel_{L^{\phi_{1}}((\varphi_{0}^{(\nu)})^{2}dm)}&\leq& 
A(CC^2(\nu,t)j^2+1+\lambda_0^{(\nu)})\int_U(\varphi_{0}^{(\nu)})^{2}\rho^{2j}\,dm\nonumber\\
&\leq&(ACC(\nu,t)+1)(C(\nu,t)+1)j^2\int_U(\varphi_{0}^{(\nu)})^{2}\rho^{2j}\,dm .
\end{eqnarray}

Having  the growth property (\ref{growth-condition}) for the function $\phi_1$ in hands, we achieve
 
\begin{equation}\label{iterate}
(\int_{U}\rho^{2j(1+\varepsilon)}(\varphi_{0}^{(\nu)})^{2}\,dm)^{\frac{1}{1+\varepsilon}}\leq
 (ACC(\nu,t)+1)(C(\nu,t)+1) j^2\int_{U}\rho^{2j}(\varphi_{0}^{(\nu)})^{2}\,dm.
\end{equation}

We iterate (\ref{iterate}).  Define $j_k=2(1+\varepsilon)^k$ for
$k=0,1...$ and 
\begin{eqnarray}
\Theta_k^{U}=\big(\int_{U}\rho^{j_k}(\varphi_{0}^{(\nu)})^{2}\,dm\big)^{\frac{1}{j_k}}\ {\rm and}\
M(\nu,t):= (ACC(\nu,t)+1)(C(\nu,t)+1).
\end{eqnarray}

Then (\ref{iterate}) can be written as

\begin{eqnarray}
\Theta_{k+1}^{U}\leq
(M(\nu,t)(1+\varepsilon)^{2k})^{\frac{1}{2(1+\varepsilon)^k}}\Theta_k^{U}.
\end{eqnarray}

Using this recursively yields 

\begin{eqnarray}
\Theta_k^{U}\leq
M(\nu,t)\Theta_0^{U}=M(\nu,t)(\int_{U}\varphi_{0}^{(\nu)})^{2}\,dm)^{\frac{1}{2}}\leq M(\nu,t).
\end{eqnarray}

for all $k=0,1,..$ .  Since the right-hand-side of the latter inequality is independent from $U$, we deduce 

\begin{eqnarray}
\displaystyle\lim_{k\rightarrow\infty}\Theta_k^{X}=
\displaystyle\sup_{X}\rho\leq M(\nu,t),
\end{eqnarray}
and this shows that
\begin{eqnarray}
\xi^{(\nu)}\leq M(\nu,t)\varphi_{0}^{(\nu)},\ \forall\,t>0,
\end{eqnarray}

which was to be proved.

\end{proof}

\begin{theo} Let $\mu$ be a positive Radon measure on Borel subsets of $X$ charging no set having zero capacity. Then under assumptions  (SUP), (ISO), (HI) and the growth condition (\ref{growth-condition}),  the following sharp estimate for the ground state $\phiomu$ holds true
\begin{eqnarray*}
\big((AC\inf_{t>0}\beta(t/2)e^{t\lamomu}+1)(\inf_{t>0}\beta(t/2)e^{t\lamomu}+1)\big)^{-1}\ximu   
\leq\phiomu\leq\ximu(\inf_{t>0}
\beta(t/2)e^{t\lamomu}),\ a.e.
\end{eqnarray*}
\label{sharp-comparison}

\end{theo}

\begin{proof} Let $\mu_k\uparrow\mu$ (as specified in Section2). By Theorem \ref{comparison1} it holds 

\begin{eqnarray}
\phiok\leq \xi_0^{(k)}\beta(t/2)e^{\lamok t},\ \forall\,t>0.
\end{eqnarray}

Now the right-hand-side inequality follows directly by letting $k\to \infty$ and using  Corollary \ref{approximation}.\\
The reversed inequality is obtained in the same manner by using Theorem \ref{comparison2} and Corollary \ref{approximation}.
\end{proof}

Let us recall that  according to Theorem \ref{compact-resolvent},   $e^{-t\Hmu}$ is Hilbert-Schmidt operator for every $t>0$. Thus $e^{-t\Hmu}$ has a $m\times m$ absolutely continuous kernel. For every $t>0$, we designate by $p_t^{\mu}$ the heat kernel of $e^{-t\Hmu}$. \\
By standard way, we deduce that the operator $\Hmu$ has a Green's kernel which we denote by $G_X^{\mu}$. We can rephrase Theorem \ref{sharp-comparison} in term of the Green's  kernel.

\begin{coro} We have

\begin{eqnarray}
\phiomu\sim\int G_X^{\mu}(\cdot,y)\,dy,\ a.e.
\end{eqnarray}
\label{end}
\end{coro}

\begin{rk}{\rm a) We immediately   derive from the latter corollary that if the Green's functions $G_X$ and $G_X^{\mu}$ are comparable then the ground states of $H$ and $\Hmu$ are comparable as well.\\
b) If $\mu$ is such that there is $C^{\mu}\in(0,1)$ with 
$$
\int_X f^2\,d\mu\leq C^{\mu}\calE[f],\ \forall\,f\in\calF,
$$
then by (\ref{reciproque}) (changing $\mu$ by $\frac{1}{C^{\mu}}\mu$),  (SUP) is satisfied. Furthermore since $\calE$ and $\calE_{\mu}$ are equivalent,   inequality (ISO) can be changed by the weaker Sobolev-Orlicz inequality
\begin{eqnarray*}
\parallel \!f^2\!\parallel_{L^{\Phi}}
\leq C_S\mathcal{E}[f],\ f\in \mathcal{F}.
\end{eqnarray*}
In this situation the compactness of $H_{\mu}^{-1}$ can be obtained directly from formula (\ref{resolvent-formula}) and we still get by Theorem \ref{comparison1} together with Lemma \ref{last-estimate}
$$
\phiomu\sim\int G_X^{\mu}(\cdot,y)\,dy,\ a.e.
$$
On the other from (\ref{resolvent-formula}), together with the proof of Lemma \ref{last-estimate} we deduce
\begin{eqnarray}
G\leq G^{\mu}\leq \frac{1}{1-C^{\mu}}G,
\end{eqnarray}
so that
\begin{eqnarray}
\phiomu\sim\psi_0.
\end{eqnarray} 
 
}
\label{weaken}
\end{rk}

We also derive by standard way the following large time asymptotics for the heat kernel.

\begin{coro} There is $T>0$ such that for every $t>T$,
\begin{eqnarray}
p_t^{\mu}(x,y)\sim e^{-\lamomu t}\phiomu(x)\phiomu(y)\sim e^{-\lamomu t}\ximu(x)\ximu(y),\,  m\times m\ a.e. .
\end{eqnarray}
It follows, in particular that
\begin{eqnarray}
 -\lamomu=\lim_{t\to\infty}\frac{1}{t}\ln\big(\frac{p_t^{\mu}(x,y)} {\ximu(x)\ximu(y)}\big).
\end{eqnarray}
\end{coro}

\begin{exa} {\rm 
Let $X$ be a bounded uniform domain in $\R^d,\ d\geq 3$ , or equivalently according to \cite[Propostions A.2,A.3]{hansen-uni-harnack} a bounded   interior non-tangentially accessible domain (which is less than NTA domain). Let $H$ be the Dirichlet-Laplacian in $L^2(X,dx)$.  With $\mu=0$  assumptions (SUP) and (ISO) (Sobolev embedding)  are fulfilled. Furthermore according to \cite[Remark 3.1]{hansen-uni-harnack}, we have
\begin{eqnarray}
G_X(x,y)\sim |x-y|^{2-d}\ a.e.\ {\rm on}\ X\times X.
\end{eqnarray}
Whence, setting $\rho(x):=dist(x,\partial X),\ \forall\,x\in X$, we derive that  there is a finite constant $c$ such that
\begin{eqnarray}
G_X(x,y)\geq c\rho(x)\rho(y),\ a.e..
\end{eqnarray}
On the other hand we have 
\begin{eqnarray}
\psi_0(x) =\lam_0^{-1}\int_X G_X(x,y)\psi(y)\,dy\geq c´\rho(x),\ \forall\,x\in X.
\label{phi-dist}
\end{eqnarray} 
Suppose  furthermore that the distance function , $\rho$ is superharmonic, i.e.,  $-\Delta\rho\geq 0$. Then setting $w=\rho^{1/2}$, a straightforward computation yields 
\begin{eqnarray}
-\Delta w-\frac{1}{4}\frac{w}{\rho^2}\geq 0 \ {\rm on}\ X.
\end{eqnarray}
Thus by a result due to Ancona \cite[Propostion 1]{ancona} we deduce
\begin{eqnarray}
\int_X\frac{f^2(x)}{\rho^2(x)}\,dx\leq 4\int_X|\nabla f|^2\,dx,\ \forall\,f\in W_0^{1,2}(X),
\end{eqnarray}  
which together with inequality (\ref{phi-dist}) imply the occurrence of Hardy inequality (HI).\\
Whence,   in this situation all assumptions are satisfied and therefore we get
\begin{eqnarray}
\psi_0\sim \int_X G_X(\cdot,y)\,dy\sim \int_X|\cdot-y|^{2-d}\,dy,\ {\rm on}\ X.
\end{eqnarray}
b) Let $\mu$ be any measure satisfying  condition of Remark \ref{weaken}-b). Then by Remark  \ref{weaken}, we get
\begin{eqnarray}
\phiomu \sim \int_X|\cdot-y|^{2-d}\,dy,\ {\rm on}\ X.
\end{eqnarray}
c)  Finally we consider the measure $\mu$ defined by $d\mu(x)=\frac{1}{4}\rho^{-2}(x)dx$. With this choice assumption  (SUP) is fulfilled by $s=\rho^{1/2}\in W_{loc}^{1,2}(X)$. Moreover,  according to \cite[Theorem 1.1-(1.2)]{filippas-mazya} (with $p=2,\ q=\frac{2d}{d-2}$) we learn  that there is $C_S\in (0,\infty)$ such that
\begin{equation}\label{hd}
  \big(\int_{X}|f|^{q}\,dx\big)^{2/q} \leq C_S\big(\calE[f]-\frac{1}{4}\int_{X}\frac{f^2}{\rho^2}\,dx\big), \ \forall\,f\in  W_0^{1,2}(X).
\end{equation}
Whence (ISO) is satisfied with $\Phi(t)=\frac{t^{q/2}}{q/2}$ which is an N-function which corresponding $\phi_1$ is admissible. Thus we establish
\begin{eqnarray}
\phiomu\sim\int_X G_X^{\mu}(\cdot,y)\,dy.
\end{eqnarray}

}
\end{exa}

{\bf Acknowledgment}: A. BenAmor is grateful for the warm hospitality during his stay at the University of Bielefeld where part of this work has  been done.

\bibliography{biblio-hardy}


\begin{tabular}{ll}
A.~Beldi &  \quad N.~Belhadjrhouma\\
Faculty of Sciences of Tunis, Tunisia &  \quad Faculty of Sciences of Tunis, Tunisia\\
{\small beldiali@gmail.com} &  \quad 
{\small nedra.belhadjrhouma@fst.rnu.tn} \\
& \\
& \\
A.~BenAmor & \\
 Faculty of Sciences of Gafsa, Tunisia & \\
{\small ali.benamor@ipeit.rnu.tn } &
\end{tabular}

\end{document}